 \documentclass[final,1p,times,number]{elsarticle}

\usepackage{amssymb}
 \usepackage{booktabs} 
 \usepackage{graphicx}
\usepackage{subcaption}

\usepackage{amsmath}
\usepackage{xcolor}

\usepackage{amsthm}
\usepackage{graphicx}
\usepackage{booktabs}      
\usepackage{subcaption}    

\theoremstyle{plain}
\newtheorem{thm}{Theorem}[section]
\newtheorem{lem}[thm]{Lemma}

\journal{Linear Algebra and its Applications}

\begin{document}


\begin{frontmatter}
\title{Computing the action of a matrix exponential on an interval via the $\star$-product approach}

\author[label1]{Stefano Pozza\corref{cor1}}
\author[label1]{Shazma Zahid\fnref{fn1}}

\cortext[cor1]{pozza@karlin.mff.cuni.cz}
\fntext[fn1]{ shazma.zahid@matfyz.cuni.cz}

\affiliation[label1]{organization={Faculty of Mathematics and Physics, Charles University},%
            addressline={Sokolovská 83}, 
            city={Prague}, 
            postcode={186 75}, 
            country={Czech Republic}}

\begin{abstract}
We present a new method for computing the action of the matrix exponential on a vector, \( e^{At}v \). 
The proposed approach efficiently evaluates the solution for all \( t \) within a prescribed bounded interval 
by expanding it into an orthogonal polynomial series. 
This method is derived from a new representation of the matrix exponential in the so-called \(\star\)-algebra, 
an algebra of bivariate distributions. 
The resulting formulation leads to a linear system equivalent to a matrix equation of Stein type, which can be solved by either direct or Krylov subspace methods. 
Numerical experiments demonstrate the accuracy and efficiency of the proposed approach in comparison to state-of-the-art techniques.
\end{abstract}


\end{frontmatter}
\section{Introduction} 
\label{sec1}
This paper introduces a new numerical method for evaluating the action of the matrix exponential, \(e^{tA}v\), for every $t$ in a given finite interval, where \(A\) is a matrix and \(v\) a vector. The approach 
exploits the fact that the solution of a linear autonomous system of ordinary differential equations (ODEs) can be expressed by a matrix exponential. In turn, the new numerical method is obtained by exploiting the new formula for the solution of such an ODE by the $\star$-product \cite{GiscardPozza,RycBouGis25} (a generalization of the Volterra convolution) followed by an appropriate discretization of the resulting expression \cite{PozzaNiel,pozza2023newk}.
The computation of matrix exponentials is a fundamental task in many areas, including quantum dynamics, control theory, network science, and the time integration of large-scale linear ODEs arising from spatial discretizations of PDEs~(see, e.g., \cite{Hochbruck2010,Ostermann2022,higham2008functions,GolMer}).  
 
 Among the most widely used approaches are the ones based on Krylov subspace methods (e.g., \cite{Moler2003}). The \texttt{expv} routine in \texttt{Expokit} \cite{Sidje} employs the Arnoldi or Lanczos process to construct a Krylov basis and approximates \(e^{tA}v\) via projection, requiring only sparse matrix–vector products and avoiding explicit construction of \(e^{tA}\). 
In \cite{al2011computing}, \texttt{expmv\_tspan} returns either a single matrix $e^{tA}B$ or a sequence of matrices $e^{t_{k}A}B$ evaluated on an equally spaced grid of points $t_k$. In the latter case, it uses the scaling phase of the scaling-and-squaring method combined with a truncated Taylor series approximation. The scaling factor and Taylor series degree are determined using the sharp truncation error bounds from \cite{al2009sharp}, which are expressed in terms of $\|A^k\|^{1/k}$ for selected values of $k$, with the norms estimated via a matrix norm estimator.

The approach proposed in this work is based on the recently developed theory of \(\star\)-algebra \cite{Pozzxa,StefanoShazma,RycBouGis25}. By discretizing the \(\star\)-product via Legendre polynomial expansions, it reformulates the problem of matrix approximation into a linear algebra problem, which can be equivalently expressed in terms of a Stein-type matrix equation. 
The \(\star\)-approach gives the solution in terms of the coefficient of the truncated Legendre polynomial expansion of \(e^{tA}v\) in a given interval \(\mathcal{I}\). Therefore, the solution can be computed for every $t \in \mathcal{I}$. This paper also provides a bound for the truncation error of the method. Combined with an Arnoldi-based Krylov projection of dimension \(k\), the method yields an efficient algorithm for large sparse matrices, allowing the solution to be evaluated at arbitrary times without recomputation.

Numerical experiments show that the proposed method can achieve accuracy comparable to that of state-of-the-art approaches. In particular, the Arnoldi-accelerated variant delivers competitive runtimes across a broad class of test problems, making it a promising alternative for large-scale exponential integrators, especially when solutions are required at many or a priori unknown time points.

The remainder of the paper is organized as follows. Section~\ref{sec2} presents the mathematical formulation of the $\star$-method, deriving the truncated Legendre expansion and the corresponding Stein equation. In Section~\ref{sec:trunc:err}, we derive an upper bound for the truncation error. Section~\ref{sec3} describes the numerical implementation of the $\star$-method, including its integration with a Krylov subspace approach, and reports detailed numerical experiments comparing it with \texttt{expv} and \texttt{expmv\_tspan}.
 Section~\ref{sec4} concludes with a discussion of potential future research.
\section{Mathematical Formulation} \label{sec2}
Consider the system of linear autonomous ODEs of the form
\begin{equation}\label{eq:matrix_ode_full}
\frac{d}{dt} \tilde{U}_s(t) = \tilde{A} \, \tilde{U}_s(t), \quad \tilde{U}_s(s) = I_N, \quad t, s \in \mathcal{I} = [-1, 1], \ t \ge s,
\end{equation}
where \(\tilde{A} \in \mathbb{C}^{N \times N}\) is a matrix and \(\tilde{U}_s(t)\) denotes the matrix-valued solution parameterized by the initial time \(s\). Note that we set the interval $\mathcal{I}=[-1,1]$ for later convenience, but this can easily be generalized to any other bounded interval. 
Introducing the notation \(\tilde{U}_s(t) := U(t,s)\), for $t\geq s$ the ODE can be expressed equally in the \(\star\)-product framework \cite{pozza2023newk,GiscardPozza,Pozzxa} as
\begin{equation}\label{eq:star_ode_full}
\frac{d}{dt} U(t,s) = A(t,s) \, U(t,s), \quad U(s,s) = I_N,
\end{equation}
where the bivariate matrix operator \(A(t,s)\) is defined by
\begin{equation} \label{heaviside}
A(t,s) := \tilde{A} \, \Theta(t-s), \quad 
\Theta(t-s) = \begin{cases} 1, & t \ge s \\ 0, & t < s \end{cases},
\end{equation}
where $\Theta(t-s)$ is known as the \emph{Heaviside step function}.
Trivially, for \(t \ge s\) we get \(\Theta(t-s) = 1\) and the solution of Eq.~\eqref{eq:star_ode_full} is the matrix exponential.
\begin{equation*}
U(t,s) = \exp\left((t-s) \tilde{A}\right).
\end{equation*}
Following \cite{pozza2023newk}, we introduce a spectral approximation framework to numerically solve Eq.~\eqref{eq:star_ode_full}.  
Let $\{p_k(t)\}_{k=0}^{M-1}$ denote the Legendre polynomials rescaled and shifted so that they are orthonormal on the interval of interest $\mathcal I=[-1,1]$, i.e.,
\[
\int_{-1}^1 p_k(\tau)p_\ell(\tau)\,d\tau=\delta_{k\ell},
\]
with $\delta_{k\ell}$ the Kronecker delta.
Since these polynomials form the basis for the analytic functions in $\mathcal I$, for a given positive integer $M$, we can approximate each entry of the solution $U(t,s)$ by a truncated double expansion.
\begin{equation}\label{matrixrep}
 U_{ij}(t,s) \approx \sum_{k=0}^{M-1} \sum_{\ell=0}^{M-1} u_{k,\ell}^{(i,j)} p_k(t) p_\ell(s),
\end{equation}
where the Legendre–Fourier coefficients are given by
\begin{equation*}\label{eq:coefficients_final}
u^{(i,j)}_{k,\ell}
   =
   \int_{\mathcal I} \int_{\mathcal I}
        U_{ij}(\tau,\rho)\,
        p_k(\tau)p_\ell(\rho)\, d\rho\, d\tau  \in \mathbb{C}.
\end{equation*}
 For every pair \((i,j)\), these coefficients are collected into the $M \times M$ matrix (block) 
\[ U^M_{i,j} :=
   \left[u_{k,\ell}^{(i,j)}\right]_{k,\ell=0}^{M-1}
   \in \mathbb C^{M\times M}.\]
Hence, defining the vector
\[
\phi_M(t):=
\begin{bmatrix}
p_0(t)\\[0.4ex]
p_2(t)\\[0.4ex]
\vdots\\[0.4ex]
p_{M-1}(t)
\end{bmatrix},
\]
the Legendre expansion of the solution element $U_{i,j}(t,s)$ can be rewritten in the matrix form
\[
 U_{ij}(t,s) \approx \sum_{k=0}^{M-1} \sum_{\ell=0}^{M-1} u_{k,\ell}^{(i,j)} p_k(t) p_\ell(s),
 = \phi_M(t)^{T} {U}_{i,j}^M \, \phi_M(s),\; t,s \in \mathcal I=[-1,1].
\]
As a consequence, we can approximate the solution of Eq.~\eqref{eq:star_ode_full} using the Kronecker product $\otimes$ by the formula
\begin{equation}\label{eq:U_Kronecker}
  U(t,s) \approx \big(I_N \otimes \phi_M^T(t)\big)\,
   \mathcal U_M\,
   \big(I_N \otimes \phi_M(s)\big),    
\end{equation}
where $I_N$ is the identity matrix of size $N$, and \(\mathcal{U}_M\) is named the \emph{coefficient matrix} of \(U(t,s)\) and is defined as the block matrix

\[
\mathcal U_M :=
\begin{bmatrix}
U^M_{1,1} & \cdots & U^M_{1,N}\\
\vdots & \ddots & \vdots\\
U^M_{N,1} & \cdots & U^M_{N,N}
\end{bmatrix}
\in \mathbb C^{NM\times NM}.
\]

Let \(T_M:=[t_{k,\ell}]_{k,\ell=0}^{M-1}\) be the \emph{coefficient matrix} of $\Theta(t-s)$, with
\begin{equation} \label{eq:coeff_heavidise}
t_{k,\ell} := \int_{-1}^1 \Theta(\tau-\rho) p_k(\tau)p_\ell(\tau)\,d\tau, \quad k=0,\dots, M-2, \quad t_{M-1,\ell} = 0,
\end{equation}
for $\ell = 0, \dots, M-1$.
and let \(\mathcal A_M := \tilde{A} \otimes T_M\). 
Note that, as suggested in \cite{pozza2023newk}, we set the last row of $T_M$ equal to zero, to reduce the errors coming from the truncation of the expansion of $\Theta(t-s)$. 
Then by the $\star$-approach outlined, e.g., in \cite{Pozzxa, pozza2023newk}, we get the approximation
\begin{equation}\label{eq:U_coeff_approx}
\mathcal U_M
   \approx
   (I_N \otimes T_M)\,
   (I_{NM} - \mathcal A_M)^{-1}.
\end{equation}
Substituting the approximation \eqref{eq:U_coeff_approx} into approximation \eqref{eq:U_Kronecker} yields
\begin{equation}\label{eq:U_Kron_T}
U(t,s)
   \approx
   \big(I_N \otimes \phi_M^T(t) T_M\big)\,
   (I_{NM} - \mathcal A_M)^{-1}\,
   \big(I_N \otimes \phi_M(s)\big).
\end{equation}
Now, given a nonzero vector $v$, consider the initial value problem 
\begin{equation}\label{eq:ivp_refined}
\frac{d}{dt} \tilde{u}(t) = \frac{1}{2} A \tilde{u}(t), \quad \tilde{u}(-1) = v, \quad t \in [-1,1],
\end{equation}
with an exact solution $\tilde{u}(t) = \exp\!\big(\frac{1}{2} A (t+1)\big) v$. Note that $\tilde{u}(t) = \exp\!\big(\frac{1}{2} A (t+1)\big) v  = \exp\!\big( A \, \hat{t}\big) v$, for $t \in [-1,1], \, \hat{t} \in [0,1]$.

Using Eq.~\eqref{eq:U_Kron_T}, the solution of Eq.\eqref{eq:ivp_refined} is approximated by
\begin{equation*}\label{eq:expAv_refined}
\exp\left(\frac{1}{2}A(t+1)\right)v \approx \left(I_N \otimes \phi_M(t)^T T_M \right) 
\left(I_{NM} - \frac{1}{2} \mathcal{A}_M\right)^{-1} 
\left(v \otimes \phi_M(s)\right).
\end{equation*}
Hence, the solution is obtained by solving the linear system:
\begin{equation}\label{eq:linsyst}
\left(I_{NM} - \frac{1}{2} \mathcal{A}_M\right)x = v \otimes \phi_M(s),
\end{equation}
or equivalently, the Stein matrix equation:
\begin{equation}\label{eq:Stein_refined}
X - \frac{1}{2} T_M X \tilde{A}^{T}
   = \phi_M(s) v^{T}, 
   \qquad X \in \mathbb{C}^{M \times N}, 
   \qquad x = \mathrm{vec}(X);
\end{equation}
see, e.g., \cite{simoncini2016computational}.
Here, $\operatorname{vec}$ denotes the vectorization of a matrix, i.e., the transformation that converts a matrix into a column vector by stacking the columns of the matrix, one after the other. 

Thanks to the $\star$-approach combined with the use of the Legendre polynomial expansion, the action of a matrix exponential on a vector (i.e., the solution of the system of ordinary differential equations in Eq.~\eqref{eq:ivp_refined}) is expressed through a linear system, or equivalently, through a matrix equation.
Note that once $X$ is computed, the solution is given by the formula
\[ \exp\left(\frac{1}{2}A(t+1)\right)v \approx \left(\phi_M(t)^T T_M X\right)^T =  X^T T_M^T \phi_M(t), \]
thanks to the properties of the Kronecker product and algebraic manipulation.

\section{Truncation error}\label{sec:trunc:err}
This section studies the truncation error of the $\star$-method, that is, we aim to bound the truncation error norm as follows:
\[ \left\| \exp(A(t+1)) - (I \otimes \phi(t)^T)(I-A \otimes T_{M})^{-1}(I \otimes e_1 \sqrt{2})  \right\| \leq C_M \left(\nu_M\right)^M,  \]
with $C_M>0$ constant (or slowly increasing) and $\nu_M<1$. The actual bound is given in Theorem~\ref{thm:bound} below, which is the objective of this section.

  Let $A$ be an $N \times N$ matrix and let $\Theta(t-s)$ and $\delta(t-s)$ denote, respectively, the Heaviside step function and the Dirac delta distribution (hereafter we will often skip the argument \((t-s)\)). 
To derive the bound, we rely on the Fréchet–Lie group structure on distributions described in \cite{RycBouGis25}, induced by the following non‑commutative product (known as the $\star$‑product).
For matrices \(F(t,s), G(t,s)\) of compatible size whose entries are suitable bivariate distributions, their $\star$‑product is defined as
\begin{equation}\label{eq:starproddef}
    F(t,s) \star G(t,s) := \int_{\mathcal{I}} F(t,\tau)\, G(\tau,s) \,\mathrm{d}\tau.
\end{equation}
In our setting, \(F\) and \(G\) are either of the form \(\tilde{H}(t,s)\Theta(t-s)\), where \(\tilde{H}(t,s)\) is analytic in both variables over \(\mathcal{I}\), or Dirac delta distributions of the form  
\(
  C\,\delta(t-s),
\)
with \(C\) a constant scalar or matrix.

We summarize below some properties of the $\star$‑product that will be used later on.  
First, the identity element is the Dirac delta, \(F \star \delta = \delta \star F = F,\) and, under suitable assumptions on \(F(t,s)\), there exists the $\star$‑inverse \(F^{-\star}\) satisfying
\(F \star F^{-\star} = \delta\),  \cite{GiscardPozza,RycBouGis25}.
For constant matrices or scalars \(C, D\), it holds:
\[
  (C \cdot F(t,s)) \star (G(t,s) \cdot D) = C\cdot (F \star G)(t,s) \cdot D;
\]
 a trivial consequence of the fact that $C$ and $D$ can be taken outside the integral in \eqref{eq:starproddef}.
Furthermore, defining the $\star$‑powers \(F^{\star k}\) as the $\star$‑product of \(F\) with itself for $k$‑times, it holds
\begin{equation}\label{eq:star:powers}
  (A \Theta(t-s))^{\star k}
    = A^k\, \Theta(t-s)^{\star k},
  \quad
  \Theta(t-s)^{\star (k+1)}
    = \frac{(t-s)^{k}}{k!}\, \Theta(t-s),
  \quad k = 0,1,\dots,
\end{equation}
see, e.g., \cite{GisPoz23}.

Using these properties, it has already been proved (\cite{giscard2015exact, GisPoz23}) that the matrix exponential can be expanded into the convergent $\star$-algebra Neumann series:
\begin{align*}
 \exp(A(t-s)) &=  \sum_{k=0}^\infty (A\Theta)^{\star k} \star \Theta = (\delta I- A\Theta)^{-\star} \star \Theta. 
\end{align*}
The following Theorem generalizes this result to the Taylor series.
\begin{thm}\label{thm:expansion}
Any matrix exponential can be expanded into a $\star$-algebra Taylor series centered at $c \in \mathbb{C}$, with $|c| < |c-1|$, that is,
\begin{align*}
 \exp(A(t-s)) &= (\delta I- A\Theta)^{-\star} \star \Theta  = \sum_{k=0}^{\infty} \frac{ \left( A \Theta - c \delta \right)^{\star k}}{(1-c)^{k+1}} \star \Theta. 
\end{align*}
Moreover, the related truncated series remainder is
\begin{align*}
 \exp(A(t-s)) - \sum_{k=0}^{\ell} \frac{ \left( A \Theta - c \delta \right)^{\star k}}{(1-c)^{k+1}} \star \Theta = \frac{1}{1-c} \sum_{k=\ell +1}^\infty \left(\frac{-c}{1-c}\right)^k L_k\left(\frac{A(t-s)}{c}\right), 
\end{align*}
where $L_k(z)$ is the Laguerre polynomial of $k$-degree. 
\end{thm}
\begin{proof}
Since $\delta$ is the identity of the $\star$-product, it holds that $(A\Theta)^{\star j} \star (-c\delta)^{\star i} = (-c\delta)^{\star i} \star (A\Theta)^{\star j}$, i.e., the factors commute. Therefore,
\begin{align*}
    \left( A \Theta - c \delta \right)^{\star k} \star \Theta &= \sum_{j=0}^k \binom{k}{j} (A\Theta)^{\star j} \star (-c\delta)^{\star k-j} \star \Theta
    = \sum_{j=0}^k \binom{k}{j} A^j \Theta^{\star j+1} (-c)^{k-j} \\
    &= (-c)^k \sum_{j=0}^k \binom{k}{j} \left(\frac{A(t-s)}{-c}\right)^j\frac{1}{j!}\Theta
    = (-c)^k L_k\left(\frac{A(t-s)}{c}\right) \Theta;
\end{align*}
See, e.g., \cite[Formula~(5.1.6)]{szego}.
Therefore, 
\begin{align*}
    \sum_{k=0}^{\infty} \frac{ \left( A \Theta - c \delta \right)^{\star k}}{(1-c)^{k+1}} \star \Theta = \frac{1}{1-c} \sum_{k=0}^\infty \left(\frac{-c}{1-c}\right)^k L_k\left(\frac{A(t-s)}{c}\right)\Theta.
\end{align*}
Since the generating function of the Laguerre polynomials \cite[Formula~(5.1.9)]{szego} is
\[
\sum_{k=0}^{\infty} L_k(z)\, y^k
= \frac{1}{1-y}\exp\!\left(-\,\frac{z\,y}{1-y}\right),
\qquad |y|<1.
\]
 Setting $y= -c/(1-c)$ and $z = A (t-s)/c$ we obtain
\begin{align}\label{eq:laguerre:series}
    \frac{1}{1-c} \sum_{k=0}^\infty \left(\frac{-c}{1-c}\right)^k L_k\left(\frac{A(t-s)}{c}\right)
    & = \exp\left(A(t-s) \right),
\end{align}
proving the theorem.
\end{proof}
As shown in \cite{Pozzxa}, for $M \rightarrow \infty$, Eq.~\eqref{eq:U_Kron_T} leads to the following exact expression
\begin{equation*}
\exp(A(t+1)) = (I \otimes \phi(t)^T) (I - A \otimes T_\infty)^{-1} (I \otimes T_\infty \phi(-1)), \quad t \in [-1, 1].
\end{equation*}
To alleviate the notation, the sizes of the vectors \(\phi(t)\) and \(e_1\) are not made explicit and will be clear from the context; here they are infinite size vectors. Note that by \eqref{eq:coeff_heavidise}, 
\[ \phi(t)^T T_\infty \,\phi(-1) = \Theta(t+1) \equiv 1, \quad t \in [-1,1].
\]
Therefore,  $T_\infty \, \phi(-1)$ is the infinite vector that contains the coefficients of the Legendre expansion of the constant function $1$, i.e.,
\begin{equation}\label{eq:e1:phi}
    T_\infty \, \phi(-1) = e_1 \sqrt{2}.
\end{equation}
Let us define the truncated series
\begin{equation}\label{eq:pl:star}
 p_\ell^\star(A)(t,s) := \sum_{k=0}^{\ell} \frac{ \left( A \Theta(t-s) - c \delta(t-s) \right)^{\star k}}{(1-c)^{k+1}}. 
 \end{equation}
Note that $p_\ell^\star(\cdot)$ is a polynomial of degree $\ell$ in the $\star$-algebra, named a $\star$-polynomial. Mapped into the ring of coefficient matrices \cite{Pozzxa}, the $\star$-polynomial becomes
\begin{equation}\label{eq:pl:mtx}
 p_\ell(A \otimes T_\infty) := \sum_{k=0}^{\ell} \frac{ \left( A \otimes T_\infty - c I \right)^{k}}{(1-c)^{k+1}},
\end{equation}
satisfying (\cite{Pozzxa})
\[p_\ell^\star(A)(t,-1) = 
 (I \otimes \phi(t)^T) p_\ell(A \otimes T_\infty)(I \otimes \phi(-1)), \quad t \in[-1,1].\]
Hence, we obtain the error expression:
\begin{align*}
    \exp(A(t\hspace{-2pt}+\hspace{-2pt}1)) -  [p_\ell^\star(A)\star \Theta](t,-1) = (I \otimes \phi(t)^T)\left[ (I - A \otimes T_\infty)^{-1} - p_\ell(A \otimes T_\infty) \right] (I \otimes e_1 \hspace{-2pt} \sqrt{2}).
\end{align*}

The following Lemma presents a bound for this error norm. 
\begin{lem}\label{lem:bound:star} Let $A$ be a diagonalizable matrix with spectral radius $\rho(A)$ and eigenvector basis $Z$, and let $p_\ell^{\star}$ be defined as in \eqref{eq:pl:star}. Then, for 
 \[|c| < |1-c|, \quad \ell > 4\rho(A)\cdot(t+1) \frac{1}{|c|}\left(\ln\left(\frac{|1-c|}{|c|}\right)\right)^{-2} - 1, \quad \gamma_\ell :=\exp\left(2\sqrt{\frac{|\lambda|(t+1)}{(\ell+1)|c|}}\right), \]
the following bound holds
\[
      \| \exp(A(t\hspace{-2pt}+\hspace{-2pt}1)) -  [p_\ell^\star(A)\star \Theta](t,-1) \| \leq \kappa(Z)  \frac{1}{|1-c|-|c|\gamma_\ell}  \left|\frac{c \gamma_\ell}{1-c}\right|^{\ell+1}, \]
where $\kappa(Z) = \| Z \| \, \| Z^{-1}\|$ is the condition number of the eigenvector basis.
\end{lem}
\begin{proof}
    Consider the eigendecomposition $A = Z \Lambda Z^{-1}$, with $\Lambda$ the diagonal matrix containing the eigenvalues $\lambda_1, \dots, \lambda_N$. Then
\begin{align*}
   \left \| \exp(A(t\hspace{-2pt}+\hspace{-2pt}1)) -  [p_\ell^\star(A)\star \Theta](t,-1) \right\| &\leq 
    \left\| Z \exp(\Lambda(t\hspace{-2pt}+\hspace{-2pt}1))Z^{-1} -  Z[p_\ell^\star(\Lambda)\star \Theta](t,-1) Z^{-1} \right\| \\
    &\leq  \kappa(Z)
    \left\| \exp(\Lambda(t\hspace{-2pt}+\hspace{-2pt}1))-[p_\ell^\star(\Lambda)\star \Theta](t,-1) \right\|, \\
    &\leq  \kappa(Z) \max_{\lambda \in \{\lambda_1, \dots, \lambda_N\}}
    \left| \exp(\lambda(t\hspace{-2pt}+\hspace{-2pt}1))-[p_\ell^\star(\lambda)\star \Theta](t,-1) \right| \\
    & =  \kappa(Z) \max_{\lambda \in \{\lambda_1, \dots, \lambda_N\}} E_\ell(\lambda).
\end{align*} 
Using Eq.~\eqref{eq:laguerre:series}, we get 
\begin{align*}
    E_\ell(\lambda) 
    &= \left| \frac{1}{1-c} \sum_{k={\ell+1}}^\infty \left(\frac{-c}{1-c}\right)^k L_k\left(\frac{\lambda(t+1)}{c}\right) \right| \leq  \frac{1}{|1-c|} \sum_{k={\ell+1}}^\infty \left| \left(\frac{-c}{1-c}\right)^k \right| \left|L_k\left(\frac{\lambda(t+1)}{c}\right) \right| \\
    &\leq  \frac{1}{|1-c|} \sum_{k={\ell+1}}^\infty \left| \left(\frac{-c}{1-c}\right)^k \right| \left| \exp\left(2 \sqrt{k\frac{|\lambda|(t+1)}{|c|}}\right) \right| \leq  \frac{1}{|1-c|} \sum_{k={\ell+1}}^\infty \left| \left( \frac{-c}{1-c} \exp\left(2\sqrt{\frac{|\lambda|(t+1)}{k|c|}}\right) \right)^k\right|,
\end{align*}
where we used the fact that
\begin{align*}
    \left|L_k(z)\right| \leq \left|\sum_{j=0}^k \binom{k}{j} \frac{(-z)^j}{j!}\right| \leq 
    \sum_{j=0}^k \frac{k^j}{(j!)^2} |z|^j = \sum_{j=0}^k \left( \frac{\left(\sqrt{k|z|}\right)^j}{j!} \right)^2 \leq \exp\left(2\sqrt{k|z|}\right).
\end{align*}
If \(|c|/|1-c| < 1\), then for
\[\ell+1 > 4|\lambda|(t+1) \left(\ln\left(\frac{|1-c|}{|c|}\right)\right)^{-2}\frac{1}{|c|},\] 
and
\[ \gamma_\ell :=\exp\left(2\sqrt{\frac{|\lambda|(t+1)}{(\ell+1)|c|}}\right) < \frac{|1-c|}{|c|}. \]
we get the following inequalities
\begin{align*}
    E_\ell(\lambda) \leq \frac{1}{|1-c|} \sum_{k={\ell+1}}^\infty \left| \left(\frac{-c}{1-c}  \gamma_\ell \right)^k\right| \leq \frac{1}{|1-c|}  \left| \left(\frac{-c \gamma_\ell}{1-c}  \right)^{\ell+1}\right| \sum_{k=0}^\infty \left| \frac{-c \gamma_\ell}{1-c} \right|^k = \left|\frac{-c \gamma_\ell}{1-c}  \right|^{\ell+1} \frac{1}{|1-c|-|c\gamma_\ell|}.
\end{align*}
\end{proof}
Note that for $\ell + 1 \geq 4 \rho(A) (t+1) / |c|$ and $e |c| < |1-c|$ we get the simpler bound:
\[
     \left \| \exp(A(t\hspace{-2pt}+\hspace{-2pt}1)) -  [p_\ell^\star(A)\star \Theta](t,-1) \right \| \leq \kappa(Z)  \frac{1}{|1-c|-|c|e}  \left|\frac{c e}{1-c}\right|^{\ell+1}, \]
which holds, in particular, for all \(1/(1-e)<c<0\).

Let us define the infinite matrix
\begin{equation}\label{eq:T_hatm}
[\hat{T}_{M}]_{k,\ell} := \left \{ \begin{array}{ll}
  [T_{\infty}]_{k,\ell}, \quad k=1, & \dots,M-1, \quad  \ell = 1, \dots, M \\
   0,  & \text{ otherwise } 
\end{array} \right. ,   
\end{equation}
 that is, the matrix obtained by setting all the elements of $T_\infty$ to zero after the $M-1$ first row and the $M$-th column. Note that \(\hat{T}_M\) is the infinite counterpart of the $M \times M$ matrix \(T_M\) defined in \eqref{eq:coeff_heavidise}. 
Since $T_\infty$ is a tridiagonal matrix, the vector $T_\infty^k\, e_1$ is zero after the $k+1$-st element. Moreover, we have the following Lemma.
\begin{lem}\label{lem:tridiag}
Let $T_\infty$ and $\hat{T}_{M}$ be the matrices defined above. Then
\begin{equation}\label{eq:lem:matching}
     T_\infty^k\, e_1 =  \hat{T}_{M}^k\, e_1, \quad k = 0, \dots, M-1.
\end{equation}
Moreover,
\begin{equation}\label{eq:lem:norm:bound}
   \left\| T_\infty^k\, e_1 \sqrt{2} \right\| \leq \frac{28}{k!} \left( \frac{9}{4} \right)^k, \quad k=0, 1, 2, \dots \,.
\end{equation}
\end{lem}
\begin{proof}
Eq.~\eqref{eq:lem:matching} is a well-known property of tridiagonal matrices. A possible proof can be given by considering $T_\infty$ and $\hat{T}_{M}$ as adjacency matrices of (weighted) graphs. Then, the equality is obtained, for example, by Lemma 5.1 in \cite{PozTud18}.
To prove Formula~\eqref{eq:lem:norm:bound}, we first recall that 
\[ \phi(t)^T(T_\infty)^k e_1 \sqrt{2} = \phi(t)^T (T_\infty)^{k+1} \phi(-1) = \frac{(t+1)^k}{k!}, \]
where we have used Eqs.~\eqref{eq:e1:phi} and \eqref{eq:star:powers}. This means that the vector
\(a^{(k)} := (T_\infty)^k e_1 \sqrt{2} \) is composed of the coefficients of the Legendre expansion of the function \((t+1)^k/k!\). Indeed, \((t+1)^k/k! = \phi(t)^T (T_\infty)^k \, T_\infty \, \phi(-1) = \phi(t)^T (T_\infty)^k e_1 \sqrt{2}\).
Consider the expansion into orthonormal Legendre polynomials
\[ (t+1)^{k} = \sum_{j=0}^k \beta_j p_j(t), \quad |\beta_j| \leq \frac{25}{9}\left( \frac{9}{4} \right)^k \left( \frac{4}{5} \right)^j \sqrt{2(2j + 1)},  \]
where the bound is obtained using  the one in \cite[Formula (2.2)]{wang} with $\rho=5/4$, rescaled for \emph{orthonormal} Legendre polynomials. Then
\[\| a^{(k)} \| \leq \frac{1}{k!} \frac{25}{9}\left( \frac{9}{4} \right)^k \sqrt{\sum_{j=0}^k 2(2j + 1) \left(\frac{4}{5}\right)^j} \leq \frac{7\sqrt{2}}{k!} \frac{25}{9} \left( \frac{9}{4} \right)^k  \leq  \frac{28}{k!} \left( \frac{9}{4} \right)^k.  \]
\end{proof}

Using Lemma~\ref{lem:tridiag} and the properties of the Kronecker product, we get the following.
\begin{align*}
  (A \otimes  T_\infty)^k\, (I \otimes e_1) &=
  (A^k \otimes  T_\infty^k\, e_1) = (A^k \otimes  \hat{T}_{M}^k\, e_1)
  = (A \otimes  \hat{T}_{M})^k\, (I \otimes e_1), \quad k = 0, \dots, M-1.
\end{align*}
Under the assumptions of Lemma~\ref{lem:bound:star}, for $M$ large enough and $\ell \leq M-1$, we get
\begin{align*}
    \left\| (I \otimes \phi(t)^T)\left[ (I - A \otimes T_\infty)^{-1} - p_\ell(A \otimes \hat{T}_{M}) \right] (I \otimes e_1 \sqrt{2}) \right\| 
    &= \left \|  \exp(A(t+1)) -  [p_\ell^\star(A)\star \Theta](t,-1) \right \| \\
    &\leq \kappa(Z)  \frac{1}{|1-c|-|c|\gamma_\ell}  \left|\frac{c \gamma_\ell}{1-c}\right|^{\ell+1}.
\end{align*}

\begin{lem}\label{lem:bound:mtx} Let $A$ be a diagonalizable matrix with spectral radius $\rho(A)$ and eigenvector basis $Z$, and let $p_\ell$ be defined as in \eqref{eq:pl:mtx}. Moreover, assume that there exists $c \in \mathbb{C}$ so that \((I-A\otimes T_{M})^{-1}\) can be expanded into a convergent Taylor series centered at $c$.
Then, for 
 \[ 9\rho(A) \left(\ln\left(\frac{|1-c|}{|c|}\right)\right)^{-2}\frac{1}{|c|} - 1 \leq \ell \leq M-2, \quad \gamma_\ell :=\exp\left(\sqrt{\frac{9\rho(A)}{(\ell+1)|c|}}\right), \]
the following bound holds
\begin{align*}
\left\|(I\hspace{-1pt}\otimes \hspace{-1pt}\phi(t)^T)\left( (I\hspace{-2pt}-\hspace{-2pt}A\hspace{-1pt}\otimes\hspace{-1pt}\hat{T}_{M})^{-1}\hspace{-4pt}-\hspace{-3pt}p_\ell(A\hspace{-1pt}\otimes\hspace{-1pt}\hat{T}_{M}) \right) (I\hspace{-1pt}\otimes\hspace{-1pt} e_1\hspace{-5pt} \sqrt{2})\right\|\hspace{-2pt} 
&\leq 28 C_M \kappa(Z) M \left| \frac{c\gamma_\ell}{1-c} \right| ^{\ell + 1}, 
\end{align*}
with \(C_M:= \| (I-A \otimes \hat{T}_{M})^{-1} \|\).
\end{lem}
\begin{proof}
The matrix \( (I - A \otimes T_{M})^{-1}  (I \otimes e_1 \sqrt{2})\) has finite size, while the matrix \( (I - A \otimes \hat{T}_{M})^{-1}  (I \otimes e_1 \sqrt{2})\) has an infinite number of rows. However, this difference is virtual, since all the additional rows of \( (I - A \otimes \hat{T}_{M})^{-1}  (I \otimes e_1 \sqrt{2})\) are equal to zero.
Therefore, the Taylor expansions of \( (I - A \otimes T_{M})^{-1}  (I \otimes e_1 \sqrt{2})\) and \( (I - A \otimes \hat{T}_{M})^{-1}  (I \otimes e_1 \sqrt{2})\) converge under the same conditions on \(c\). Then, the well-known remainder of the Taylor expansion gives
\begin{align*}
\left( (I\hspace{-2pt}-\hspace{-2pt}A\hspace{-1pt}\otimes\hspace{-1pt}\hat{T}_{M})^{-1}\hspace{-4pt}-\hspace{-3pt}p_\ell(A\hspace{-1pt}\otimes\hspace{-1pt}\hat{T}_{M}) \right) (I\hspace{-1pt}\otimes\hspace{-1pt} e_1\hspace{-5pt} \sqrt{2})\hspace{-2pt} &=  (I-A \otimes \hat{T}_{M})^{-1}   \left( \frac{A\otimes \hat{T}_{M} - cI}{1-c} \right)^{\ell+1} (I \otimes e_1 \sqrt{2}) \\
&=  (I-A \otimes \hat{T}_{M})^{-1}   \left( \frac{A\otimes {T}_{\infty} - cI}{1-c} \right)^{\ell+1} (I \otimes e_1 \sqrt{2}) \\
&=  \frac{(I\hspace{-2pt}-\hspace{-2pt}A\hspace{-1pt}\otimes\hspace{-1pt}\hat{T}_{M})^{-1}}{(1-c)^{\ell+1}} \hspace{-2pt} \sum_{k=0}^{\ell+1} \binom{\ell+1}{k} \left( A^k\otimes {T}_{\infty}^k e_1\hspace{-5pt} \sqrt{2}\right) (-c)^{\ell+1-k} 
\end{align*}
Moreover, denoting by \(\hat{\phi}^{(M)}(t)\) the vector obtained by setting to zero all elements after the $M$-th one in $\phi(t)$, we obtain
\begin{align*} 
    (I\hspace{-1pt}\otimes \hspace{-1pt}\hat\phi^{(M)}(t)^T) (I\hspace{-2pt}-\hspace{-2pt}A\hspace{-1pt}\otimes\hspace{-1pt}\hat{T}_{M})^{-1}(I\hspace{-1pt}\otimes\hspace{-1pt} e_1\hspace{-5pt} \sqrt{2}) 
    &= (I\hspace{-1pt}\otimes \hspace{-1pt}\phi(t)^T) (I\hspace{-2pt}-\hspace{-2pt}A\hspace{-1pt}\otimes\hspace{-1pt}\hat{T}_{M})^{-1}  (I\hspace{-1pt}\otimes\hspace{-1pt} e_1\hspace{-5pt} \sqrt{2}) \\ 
    (I\hspace{-1pt}\otimes \hspace{-1pt}\hat\phi^{(M)}(t)^T)p_\ell(A\hspace{-1pt}\otimes\hspace{-1pt}\hat{T}_{M}) (I\hspace{-1pt}\otimes\hspace{-1pt} e_1\hspace{-5pt} \sqrt{2}) 
    &= (I\hspace{-1pt}\otimes \hspace{-1pt}\phi(t)^T) p_\ell(A\hspace{-1pt}\otimes\hspace{-1pt}\hat{T}_{M})(I\hspace{-1pt}\otimes\hspace{-1pt} e_1\hspace{-5pt} \sqrt{2}),
\end{align*}
where the second equality holds since $\ell \leq M-2$. As a consequence, 
\begin{align*}
G_\ell(t) &:= (I\hspace{-1pt}\otimes \hspace{-1pt}\phi(t)^T)\left( (I\hspace{-2pt}-\hspace{-2pt}A\hspace{-1pt}\otimes\hspace{-1pt}\hat{T}_{M})^{-1}\hspace{-4pt}-\hspace{-3pt}p_\ell(A\hspace{-1pt}\otimes\hspace{-1pt}\hat{T}_{M}) \right) (I\hspace{-1pt}\otimes\hspace{-1pt} e_1\hspace{-5pt} \sqrt{2}) \\ 
&=     (I\hspace{-1pt}\otimes \hspace{-1pt}\hat{\phi}^{(M)}(t)^T)\left( (I\hspace{-2pt}-\hspace{-2pt}A\hspace{-1pt}\otimes\hspace{-1pt}\hat{T}_{M})^{-1}\hspace{-4pt}-\hspace{-3pt}p_\ell(A\hspace{-1pt}\otimes\hspace{-1pt}\hat{T}_{M}) \right) (I\hspace{-1pt}\otimes\hspace{-1pt} e_1\hspace{-5pt} \sqrt{2}).
\end{align*}
Thus, combining the previous relations with Lemma~\ref{lem:tridiag} gives
\begin{align*}
\left\|G_\ell(t)\right\|
&\leq  \frac{\| (I \otimes \hat\phi^{(M)}(t)^T) (I-A \otimes \hat{T}_{M})^{-1} \|}{|(1-c)^{\ell+1}|}  \sum_{k=0}^{\ell+1} \binom{\ell+1}{k} \left\| A^k\otimes {T}_{\infty}^k e_1 \sqrt{2}\right\| |c|^{\ell+1-k} \\
&\leq  \frac{\| I \otimes \hat\phi^{(M)}(t)^T\| \, C_M}{|(1-c)^{\ell+1}|}  \sum_{k=0}^{\ell+1} \binom{\ell+1}{k} \left\| A^k \right\| \left\| {T}_{\infty}^k  e_1 \sqrt{2} \right\| |c|^{\ell+1-k} \\
&\leq  \frac{\|\hat\phi^{(M)}(t) \|C_M}{|(1-c)^{\ell+1}|}  \sum_{k=0}^{\ell+1} \binom{\ell+1}{k} \left\| A \right\|^k  \frac{28}{k!} \left( \frac{9}{4} \right)^k |c|^{\ell+1-k} \\
&\leq 28 M C_M \kappa(Z) \frac{|c|^{\ell+1}}{|(1-c)^{\ell+1}|}  \sum_{k=0}^{\ell+1} \binom{\ell+1}{k}   \frac{1}{k!} \left( \frac{9 \rho(A)}{4 |c|} \right)^k \\
&\leq 28 \kappa(Z) C_M M \left( \left| \frac{c}{(1-c)} \right| \ \exp \left( \sqrt{\frac{9 \rho(A)}{(\ell + 1)|c|}} \right) \right)^{\ell + 1},
\end{align*}
where we used the fact that \( \| \phi^{(M)}(t) \| \leq M\) since the $j$-th element of the vector is bounded by \(|[\phi^{(M)}(t)]_j| \leq \sqrt{2j+1}\), \(t \in [-1,1]\).
Similarly to the proof of Lemma~\ref{lem:bound:star}, for
\[\ell+1 > 9\rho(A) \left(\ln\left(\frac{|1-c|}{|c|}\right)\right)^{-2}\frac{1}{|c|}, \quad\gamma_\ell :=\exp\left(\sqrt{\frac{9\rho(A)}{(\ell+1)|c|}}\right), \] 
 it holds
\begin{align*}
\left\|(I\hspace{-1pt}\otimes \hspace{-1pt}\phi(t)^T)\left( (I\hspace{-2pt}-\hspace{-2pt}A\hspace{-1pt}\otimes\hspace{-1pt}\hat{T}_{M})^{-1}\hspace{-4pt}-\hspace{-3pt}p_\ell(A\hspace{-1pt}\otimes\hspace{-1pt}\hat{T}_{M}) \right) (I\hspace{-1pt}\otimes\hspace{-1pt} e_1\hspace{-5pt} \sqrt{2})\right\|\hspace{-2pt} 
&\leq 28 C_{M} \kappa(Z) M \left| \frac{c\gamma_\ell}{1-c} \right| ^{\ell + 1}, 
\end{align*}
concluding the proof.
\end{proof}

\begin{thm}\label{thm:bound}
Let $A$ be a diagonalizable matrix with spectral radius $\rho(A)$ and eigenvector basis $Z$ and
assume that there exists $c \in \mathbb{C}$ so that \((I-A\otimes T_{M})^{-1}\) can be expanded into a convergent Taylor series centered at $c$. Then, for 
\[M  \geq 9\rho(A) \frac{1}{|c|}\left(\ln\left(\frac{|1-c|}{|c|}\right)\right)^{-2} + 1, \quad \mu_{M} :=\exp\left(\sqrt{\frac{9\rho(A)}{(M-1)|c|}}\right),\]
the following bound holds
\[ \left\| \exp(A(t+1)) - (I \otimes \phi(t)^T)(I-A \otimes T_{M})^{-1}(I \otimes e_1 \sqrt{2})  \right\| \leq \kappa(Z) K_M M \left| \frac{c\mu_{M}}{1-c} \right|^{M-1},  \]
with \( K_M :=\left| (|1-c|-|c|\mu_{M})M \right|^{-1} + 28 \| (I- A \otimes \hat{T}_M)^{-1} \| \). Moreover, if \(c = - 9 \rho(A)/(M-1)\) is an admissible value, then for \(M \geq 18 \rho(A)+1\) the bound becomes:
\[ \left\| \exp(A(t+1)) - (I \otimes \phi(t)^T)(I-A \otimes T_{M})^{-1}(I \otimes e_1 \sqrt{2})  \right\| \leq \kappa(Z) K_M M \left| \frac{9\rho(A)e}{M-1} \right|^{M-1},   \]
which describes a super-exponential decay.
\end{thm}
\begin{proof}
Let us define the following errors
\begin{align*}
    E(t) &:=  \exp(A(t+1)) - (I \otimes \phi(t)^T) (I - A \otimes T_{M})^{-1}  (I \otimes e_1 \sqrt{2}),
     \\
    F_\ell(t)  &:=  (I \otimes \phi(t)^T)\left[ (I - A \otimes T_\infty)^{-1} - p_\ell(A \otimes \hat{T}_{M}) \right] (I \otimes e_1 \sqrt{2}),  \\
    G_\ell(t) &:= (I \otimes \phi(t)^T) \left( (I-A\otimes\hat{T}_{M})^{-1}-p_\ell(A\otimes\hat{T}_{M}) \right) (I\otimes e_1\sqrt{2}) .
\end{align*}
As already pointed out, the matrix \( (I - A \otimes T_{M})^{-1}  (I \otimes e_1 \sqrt{2})\) has a finite size, while the matrix \( (I - A \otimes \hat{T}_{M})^{-1}  (I \otimes e_1 \sqrt{2})\) has an infinite number of rows. However, this difference is virtual, since all the additional rows of \( (I - A \otimes \hat{T}_{M})^{-1}  (I \otimes e_1 \sqrt{2})\) are equal to zero. Therefore,
    \begin{align*}
    \| E(t) \| &= \left\|  (I \otimes \phi(t)^T)\left[ (I - A \otimes T_\infty)^{-1}  - (I - A \otimes \hat{T}_{M})^{-1} \right] (I \otimes e_1 \sqrt{2}) \right\| \\
    &= \left\|  (I \otimes \phi(t)^T)\left[ (I - A \otimes T_\infty)^{-1} - p_\ell(A \otimes \hat{T}_{M}) + p_\ell(A \otimes \hat{T}_{M})  - (I - A \otimes \hat{T}_{M})^{-1} \right] (I \otimes e_1 \sqrt{2}) \right\| \\
    & = \left\| F_\ell(t) - G_\ell(t) \right\|.
    \end{align*}
    By Lemmas~\ref{lem:bound:star} and \ref{lem:bound:mtx}, 
     for
\[ 9\rho(A) \left(\ln\left(\frac{|1-c|}{|c|}\right)\right)^{-2}\frac{1}{|c|}-1 \leq \ell \leq M - 2, \quad  \mu_{\ell+2} :=\exp\left(\sqrt{\frac{9\rho(A)}{(\ell+1)|c|}}\right),  \] 
    we obtain
    \begin{align*}
        \| E(t) \| &\leq \| F_\ell(t) \| + \| G_\ell(t) \| \leq \kappa(Z)  \frac{1}{|1-c|-|c|\mu_{\ell+2}}  \left|\frac{c \mu_{\ell+2}}{1-c}\right|^{\ell+1} + 28 C_{M} \kappa(Z) M \left| \frac{c \mu_{\ell+2}}{1-c} \right| ^{\ell + 1} \\
        &\leq \kappa(Z) \left( \frac{1}{|1-c|-|c|\mu_{\ell+2}}\frac{1}{M} + 28 C_{M} \right) M \left| \frac{c \mu_{\ell+2}}{1-c} \right| ^{\ell + 1} \leq \kappa(Z) K_M M \left| \frac{c \mu_{\ell+2}}{1-c} \right| ^{\ell + 1}.
    \end{align*}
    
    Finally, assuming that the Taylor expansion centered at \(c = -9\rho(A)/(M-1)\) converges and setting \(\ell = M-2\), we get
    \begin{align*}
        \| E(t) \| &\leq \kappa(Z)K_M M \left(\frac{\mu_{M}}{M-1} \right)^{M-1} \leq \kappa(Z)K_M M\left(\frac{9 \rho(A) e}{M-1} \right)^{M-1} , 
    \end{align*}
    note that for \(M \geq 18 \rho(A)+1\) the condition
    \begin{align*}
         M  \geq 
         9\rho(A) \left(\ln\left(\frac{|1-c|}{|c|}\right)\right)^{-2}\frac{1}{|c|} +1 = (M-1)\left(\ln\left(1 + \frac{M-1}{9 \rho(A)}\right)\right)^{-2}+1,
    \end{align*}
    is satisfied.
\end{proof}

The Theorem holds under the 
assumption that \((I-A\otimes T_{M})^{-1}\) can be expanded into a convergent Taylor series centered on $c$. However, this is not always possible; for example, when \(A\otimes T_{M}\) has real eigenvalues larger and smaller than $1$. In general, a sufficient condition is \(\rho(A)\rho(T_M) <1\). However, something more can be said using the following Lemma.  

\begin{lem}
    The eigenvalues of the truncated matrix $T_M$ have a nonnegative real part.
\end{lem}
\begin{proof}
  Consider the matrix
    \[ \tilde{T}_M = \begin{pmatrix}
1 & -\alpha_1 & 0 & 0 & \cdots & 0 \\
\alpha_1 & 0 & -\alpha_2 & 0 & \cdots & 0 \\
0 & \alpha_2 & 0 & -\alpha_3 & \ddots & \vdots \\
\vdots & \ddots & \ddots & \ddots & \ddots & 0 \\
0 & \cdots & 0 & \alpha_{M-2} & 0 & -\alpha_{M-1} \\
0 & \cdots & \cdots & 0 & \alpha_{M-1} & 0
\end{pmatrix},
\qquad
\alpha_k = \frac{1}{\sqrt{(k+2)\,k}},\]
which can be decomposed into
\( \tilde{T}_M = E_1 + K_M,\)
with $K_M$ skew and $E_1=e_1 e_1^T$. 
Let $\lambda$ be an eigenvalue of $\tilde{T}_M$ with $v$ a related eigenvector with \(\|v\|^2 = 1\). Then
\begin{align*}
    \lambda  = v^H \tilde{T}_M \, v  = v^H E_1 \,v + v^H K_M \, v.
\end{align*}
Clearly $v^H E_1 v \geq 0$ while $v^H K_M v$ is purely imaginary. Hence, the real part of $\lambda$ is nonnegative.
Note that, according to the results in \cite{pozza2023newk}, the matrix $T_M$ is equal to $\tilde{T}_M$ but for the last row, where it is zero. 
Therefore, the eigenvalues of $T_M$ are $0$ and all the eigenvalues of $\tilde{T}_{M-1}$, concluding the proof.
\end{proof}

It is possible to rewrite the matrix exponential in the more convenient form:
\[\exp(A(t+1)) = \exp(\alpha I (t+1)) \exp((A-\alpha I)(t+1)).\]
Denote the eigenvalues of the matrices $A$ and $T_M$, respectively, by $\lambda_i$ and $\theta_j$. Then the eigenvalues of \((A-\alpha I)\otimes T_M\) are of the form \( (\lambda_i - \alpha) \theta_j\). Since all $\theta_j$ have a positive real part, there exists a \(\alpha\) large enough, so that all the eigenvalues of \((A-\alpha I)\otimes T_M \) have a negative real part. Therefore, the resolvent \((I-(A-\alpha I)\otimes T_M)^{-1}\) can be expanded into a Taylor series for every $c<0$ leading to a super-exponential decay.

Finally, we remark that the Taylor expansion is generally not the best approximation for a matrix inverse. Moreover, when the problem is \(\exp(A(t+1))\, v\), with \(v\) a given vector, other polynomial approaches can be used, e.g., Krylov subspace methods. Thus, the fact that the Taylor series might not converge should be regarded as a failure of the presented bound and not a condition under which the $\star$-method necessarily fails. Other polynomial approaches might lead to better bounds.

\section{Numerical Methods}\label{sec3}
Numerous numerical methods have been proposed in the literature to solve matrix equations of form~\eqref{eq:Stein_refined}. We refer to \cite{simoncini2016computational} for a survey of the principal approaches. In this paper, we consider the standard Schur decomposition method, and we then combine it with a Krylov subspace approach for large and sparse matrices.

We compare our approach with two well-established routines: \texttt{expv} \cite{Sidje}, which computes an approximation of~$\exp(A)v$ for a general matrix~$A$ using Krylov subspace projection techniques without explicitly forming the matrix exponential, and \texttt{expmv\_tspan}~\cite{al2011computing}, which can return either a single matrix~$e^{tA}B$ or a sequence~$e^{t_kA}B$ on an equally spaced grid of points~$t_k$. The latter employs the scaling phase of the scaling-and-squaring method together with a truncated Taylor series approximation, determining the scaling factor and Taylor degree via the sharp truncation error bounds of~\cite{al2009sharp}, expressed in terms of~$\|A^k\|^{1/k}$ for selected values of~$k$, with norms estimated using a matrix norm estimator.
\subsection{Algorithms}\label{subsec1}
The first method to solve Eq.~\eqref{eq:ivp_refined} is based on the \emph{Schur's decomposition} of the \emph{coefficient matrix} associated with $\Theta(t-s)$, i.e., $T_M = USU^{H},$
where $U$ is unitary and $S$ is a upper triangular matrix.  
With the change of variables $Y=U^{H}X$ and $w=U^{H}\phi_M(-1)$, Eq.~\eqref{eq:Stein_refined}
becomes
\begin{equation}\label{eq:stein_schur}
    Y - \frac12 S Y \tilde{A}^{T} = w v^{T}.
\end{equation} 
Schur decomposition is performed only once, stored, and reused; the overall computational cost is therefore dominated by solving~\eqref{eq:stein_schur}.
We consider two variants for computing the solution:
\begin{itemize}
  \item[\textbf{1.}]\textbf{$\star$-method.}  
    \begin{itemize}
        \item Solve the matrix equation~\eqref{eq:stein_schur} using the MatLab R2022b function \texttt{dlyap} from the Control System Toolbox. Note that \texttt{dlyap} does not compute the Schur decomposition of $S$ since it is already triangular.
        \item Return the matrix $X = UY.$ Now, for every $t \in [-1, 1]$ the solution is accessible by the formula \( \exp\left(\frac{1}{2}\tilde{A}(t+1)\right)v \approx   X^T T_M^T \phi_M(t)  \).
    \end{itemize}

  \item[\textbf{2.}]\textbf{$\star$-method + Arnoldi.}  
    \begin{itemize}
        \item Perform $k$ iterations of the Arnoldi algorithm with inputs $\tilde{A}$ and $v$, producing an orthonormal basis $V_k$ that spans the Krylov subspace $
            \mathcal{K}_k(\tilde{A},v) = \operatorname{span}\{v, \tilde{A}v, \ldots, \tilde{A}^{k-1}v\}$, and the related upper Hessenberg matrix $H_k$. Project~\eqref{eq:stein_schur} onto this subspace yields
        \begin{equation}\label{eq:stein_reduced}
            Z - \frac{1}{2}S Z H_k^{T} = w e_1^{T} \| v \|,
        \end{equation}
        with $Y \approx Z V_k^T$, note that $Z$ has now $k$ columns, a reduced size.
        \item Solve Eq.~\eqref{eq:stein_reduced} using the \texttt{dlyap} function of MatLab.
        \item Return the matrix $X = U Y V_k^T$. Now, for every $t \in [-1, 1]$ the solution is accessible by the formula \( \exp\left(\frac{1}{2}{\tilde{A}}(t+1)\right)v \approx  X^T T_M^T \phi_M(t)  \).
    \end{itemize}
\end{itemize}
\subsection{Numerical Tests}
We evaluated the performance of the proposed $\star$-methods (Section~\ref{subsec1}) using a set of matrices from the literature, as described below. All computations were carried out in MatLab R2022b on machines with an 11th Gen Intel\textsuperscript{\textregistered} Core\texttrademark~i7 processor. Performance is measured using two criteria: \emph{accuracy}, measured as the average relative error with respect to a reference solution computed via MatLab's \texttt{expm}, and \emph{computational cost}, using MatLab's \texttt{tic}/\texttt{toc}. 

The following test matrices are used in the experiments:
\begin{itemize}
\item \textbf{2D Poisson matrix:} 
This experiment is a variation of one from \cite{al2011computing, trefethen2006talbot}. The matrix \(\tilde{A}\in\mathbb{R}^{2500\times 2500}\) is the standard finite difference discretization of the $2D$ Laplacian with opposite sign, namely the block tridiagonal matrix constructed by \verb|-gallery('poisson',n)| in MatLab on a specific (normalized) vector.

\item \textbf{Complex tridiagonal matrix:} A tridiagonal matrix with diagonal entries $2i$ and off-diagonal entries $-i$, with small perturbations ($\epsilon = 10^{-13}$) at the extreme top left and bottom right elements. Vector $v$ is the first standard unit basis vector.

\item \textbf{Matrices with decaying eigenvalues:}
The matrix is constructed as $\tilde{A} = Q\Lambda Q^T$, where 
$\Lambda = \mathrm{diag}(\lambda_1, \ldots, \lambda_n)$ with 
$\lambda_i = \exp\!\left(-\frac{5(i-1)}{n-1}\right)$ for $i = 1, \ldots, n$, 
and $Q$ is the orthogonal factor from the QR factorization of a random $n\times n$ matrix. 
The right-hand side $v$ is a random normalized vector.
    
\item \textbf{Tridiagonal Toeplitz matrix:} This experiment is taken from \cite{strang2014functions}, where we consider a symmetric tridiagonal Toeplitz matrix generated by the MatLab command \verb|toeplitz(c,r)|, and $v$ is a normalized random vector. Here, 
$c = [2, -1, 0, \dots, 0]^{T}, r = [2, -1, 0,\dots, 0].$

 \item \textbf{Pentadiagonal Toeplitz matrix }  
This experiment considers a variation of the matrix studied in \cite{aprahamian2014matrix}, generated using the MatLab command  
\verb|gallery('toeppen', n)|. The vector $v$ is a random normalized vector.
\item \textbf{Dense matrix with eigenvalues at Chebyshev nodes:} 
The matrix is constructed as $\tilde{A} = Q\Lambda Q^T$, where 
$\Lambda =  \mathrm{diag}(\lambda_1, \ldots, \lambda_n)$ with 
$\lambda_i = \cos\bigl((2i-1)\pi/(2n)\bigr)$ for $i = 1, \ldots, n$, 
and $Q$ is the orthogonal factor from the QR factorization of a random $n\times n$ matrix. 
The right-hand side $v$ is a random normalized vector.

\item \textbf{Matrix from the FEM semidiscretization of a Schrödinger equation}
In  our final experiment, we consider the Schr\"{o}dinger equation of the form:
\begin{equation*}
i \psi'(x,t) = -\frac{\hbar}{2\hat{m}} \Delta \psi(x,t) + V(x) \psi(x), \; \psi(x,0)  = \psi_0(x), \; x \in \Omega = [-2, 2] \times [-2, 2],
\end{equation*}
for \( t \in [0,t_\text{max}]\).
 The potential \(V(x)\) is defined as \(V(x) = 0\) for \(x \in \mathrm{int}(\Omega_p)\) and \(V(x) = 10\) for \(x \in \mathrm{cl}(\Omega / \Omega_p)\). The initial value $\psi_0(x)$ is \(\psi_0(x) = e^{-|x|^2 / 2}\). Without loss of generality, we set \(\hbar  = \hat{m} = t_{max} = 1\) for the sake of simplicity.
 We utilize quadratic Lagrangian elements and set Dirichlet boundary conditions, which are enforced using the \emph{stiff-spring} method.
With this approach, we obtain a system of ODEs of the form  
\[
\mathsf{M}_n \boldsymbol{\psi}'(t) = \mathsf{K}_n \boldsymbol{\psi}(t), \quad  
\boldsymbol{\psi}(0) = \boldsymbol{\psi}_0, \quad t \in (0,t_\text{max}]
\] 
whose solution is accessible as a matrix exponential $\psi(t) = e^{A t}\psi_0$, with $A = M_n^{-1}K_n$.
\end{itemize}
In all the experiments mentioned above, we have multiplied the test matrices by time $t_{\max}$.

We compare the proposed $\star$-method (Section~\ref{subsec1}) with the \texttt{expv\_tspan} algorithm \cite{al2011computing}.
In our experiments, \texttt{expv\_tspan} evaluates $\exp(t_k A)v$ at $q+1$ uniformly spaced points $t_k \in [t_0, t_{\max}]$.
In contrast, the $\star$-method does not require specifying the points at which the solution is computed, since the solution is given through the coefficients of its Legendre polynomial expansion. However, it does require setting the truncation parameter $M$ (see Eq.~\eqref{matrixrep}).
In our tests, we select $M$ to be as small as possible. We assume that this value is close to the minimal number of points required to interpolate the solution on $[t_0, t_{\max}]$ (noting that the Legendre polynomial expansion exhibits behavior similar to Chebyshev polynomials). Then, we establish $q = M$ as a simple yet fair criterion for comparing the two methods.

The remaining main parameters in the experiments are the matrix size $n$, initial time $t_0$, final time $t_{\max}$, number of Arnoldi iterations $k$. All parameters used in the experiments are listed in detail in Table~\ref{tab:parameters}, with the corresponding average relative errors in Table~\ref{tab:relative_error}, and the average computational times in Table~\ref{tab:computational_time}.
 \begin{table}[t]
 \centering
\caption{Parameters used in the numerical experiments. Here, $n$ is the matrix size, $t_{\max}$ is the final time, $k$ is the number of Arnoldi iterations, $M$ is the truncation parameter, and $q$ denotes the number of time steps in \cite{al2011computing}.}

    \label{tab:parameters}
    \begin{tabular}{llccccc}
        \toprule
        \textbf{Id} & \text{Description} & \( n \) & \( t_{\text{max}} \) & \( q=M\) & \( k \) \\
        \midrule
        1 & 2D Poisson matrix & 2500 & 4 & 22 & 35 \\
        2 & Complex tridiagonal matrix & 1002 & 8 & 7 & 17 \\
        3 & Matrix with decaying eigenvalues  & 2000 & 4 & 13 & 17 \\
        4 & Matrix with decaying eigenvalues   & 20 & 4 & 12 & 19 \\
        5 & Tridiagonal Toeplitz matrix & 100 & 4 & 25 & 22 \\
         6 & Pentadiagonal Toeplitz matrix & 1000 & 2 & 38 & 80 \\
        7 & Dense matrix with eigenvalues at Chebyshev nodes & 500 & 4 & 12 & 20 \\
    
        8 & FEM semidiscretization of a Schr\"{o}dinger equation & 841 & 0.1 & 50 & 70 \\
        9 & FEM semidiscretization of a Schr\"{o}dinger equation & 1917 & 0.1 & 150 & 120 \\ 
        10 & FEM semidiscretization of a Schr\"{o}dinger equation & 3437 & 0.1 & 190 & 180 \\
        \bottomrule
    \end{tabular}
\end{table}

\begin{table}[t]
\centering
\caption{Average relative error for the considered methods across the examples in Table~\ref{tab:parameters}.}
\label{tab:relative_error}
\begin{tabular}{lcccccc}
\toprule
\textbf{Method} & 1 & 2 & 3 & 4 & 5  \\ 
\midrule
\texttt{expv}                     & 6.2391e-15 & 2.6642e-13 & 2.5671e-15 & 8.4500e-16 & 6.1633e-15 \\ 
\texttt{expmv\_tspan}             & 6.5871e-15 & 9.7714e-15 & 1.4853e-15 & 7.5326e-16 & 6.4114e-15 \\ 
\texttt{$\star$-method}           & 6.1289e-15 & 7.9682e-14 &            & 7.2386e-15 & 9.8779e-10 \\ 
\texttt{expv\_tspan + Arnoldi}    & 6.6221e-15 & 4.3728e-14 & 5.8626e-15 & 9.3941e-16 & 5.9311e-15 \\ 
\texttt{$\star$-method + Arnoldi} & 6.6942e-15 & 7.4874e-14 & 6.3234e-15 & 9.5022e-15 & 2.8513e-10 \\ 
\midrule
 & 6 & 7 & 8 & 9 & 10  \\ 
\midrule
\texttt{expv}                     & 6.7021e-11 & 2.7977e-15 & 1.6186e-10 & 2.0248e-10 &  1.7658e-10 \\ 
\texttt{expmv\_tspan}             & 3.1677e-15 & 2.0645e-15 & 3.9387e-14 & 5.9559e-14 &  1.2115e-13 \\ 
\texttt{$\star$-method}           & 4.125e-14 & 2.4825e-14 &  3.9788e-13 & 1.6147e-12 & 2.7522e-12  \\ 
\texttt{expv\_tspan + Arnoldi}    & 1.2518e-15 & 7.9543e-14 & 1.8857e-13 & 1.3276e-12 & 7.7015e-13 \\ 
\texttt{$\star$-method + Arnoldi} & 2.202e-14 & 8.0757e-14 & 3.7302e-14 & 6.2359e-14 & 1.3388e-13 \\ 
        \bottomrule
\end{tabular}
\end{table}

\begin{table}[t]
\centering
\caption{Average computational time (in seconds) for the experiments in Table~\ref{tab:relative_error}.}
\label{tab:computational_time}
\begin{tabular}{lcccccc}
\toprule
\textbf{Method} & 1 & 2 & 3 & 4 & 5  \\ 
\midrule
\texttt{expv}                     & 0.00715 & 0.00016 & 0.14267 & 0.00018 & 0.00245 \\ 
\texttt{expmv\_tspan}             & 0.00732 & 0.01147 & 0.21265 & 0.00118 & 0.00360 \\  
\texttt{$\star$-method}           & 3.2695 &  2.7069  &               & 0.00053 & 0.00176 \\ 
\texttt{Arnoldi}                  & 0.00363 & 0.00482 & 0.03668 & 0.00013 & 0.00063 \\
\texttt{expv\_tspan + Arnoldi}    & 0.00573 & 0.00819 & 0.03840 & 0.00156 & 0.00229 \\  
\texttt{$\star$-method + Arnoldi} & 0.00520 & 0.00564 & 0.03716 & 0.00057 &  0.00211 \\  
\midrule
 & 6 & 7 & 8 & 9 & 10  \\ 
\midrule
\texttt{expv}                     & 0.00654 & 0.00405 & 0.30268 & 2.5333 & 13.1203 \\
\texttt{expmv\_tspan}             & 0.00700 & 0.00665 & 0.87762 & 7.858 & 40.9493 \\    
\texttt{$\star$-method}           & 2.0619  & 0.2244  & 1.3857 & 11.9798 & 61.6381 \\
\texttt{Arnoldi}                  & 0.00833 & 0.00092 & 0.17646 & 1.4437 & 6.6184 \\
\texttt{expv\_tspan + Arnoldi}    & 0.01079 & 0.00256 & 0.18837 & 1.4787 & 6.6919 \\  
\texttt{$\star$-method + Arnoldi} & 0.01897 & 0.00122 & 0.18358 & 1.4776 & 6.6859 \\  
        \bottomrule
\end{tabular}
\end{table}

We ran each of the described methods 100 times in the examples listed in Table~\ref{tab:parameters}. Table~\ref{tab:relative_error} shows the average relative errors at the final time $t_\text{max}$ for each test matrix. Both the $\star$-method and the $\star$-method combined with the Arnoldi algorithm achieve results comparable in accuracy to $\texttt{expv\_tspan}$.
The only exceptions are the third experiment, where the $\star$-Krylov method was too computationally expensive to execute, and the fifth experiment, where the $\star$-method fails to achieve relative errors smaller than $\mathcal{O}(10^{-10})$, even for larger values of $M$.
This large relative error can be explained by noting that the condition number of the matrix
$I - \tilde{A} t_{\max}/2 \otimes S$,
is of order \(10^{7}\). According to standard perturbation estimates (see, for example, \cite[Chapter~16]{Hig02book}), such a condition number can lead to a relative error of approximately the magnitude observed in the numerical results.

Table~\ref{tab:computational_time} reports the corresponding average computational times. While the $\star$-method combined with the Arnoldi algorithm performs comparably to, and occasionally faster than, the Arnoldi-based \texttt{expv\_tspan}, the plain $\star$-method (without a Krylov-based approach) becomes prohibitively expensive for large matrices, as expected.
It is also worth noting that although \texttt{expv} appears faster than the other methods, it computes the solution only at the final time point $t_{max}$. Therefore, comparing it directly to the others is not entirely fair. We include its runtime solely to highlight why alternative methods should be considered when solutions are required at multiple time points or across the entire interval.
\begin{table}[t]
\centering
\caption{Relative error versus truncation parameter $M$ for Example 1 ($n=50$, $k=35$, $t_{\max}=4$).}
\label{tab:errors_vs_M_example1}
\begin{tabular}{cc|cc}
\toprule
$M$ & Relative Error & $M$ & Relative Error \\
\midrule
10 & $4.23 \times 10^{-6}$  & 18 & $4.27 \times 10^{-12}$ \\
11 & $9.80 \times 10^{-7}$  & 19 & $5.39 \times 10^{-13}$ \\
12 & $2.12 \times 10^{-7}$  & 20 & $6.37 \times 10^{-14}$ \\
13 & $4.26 \times 10^{-8}$  & 21 & $1.49 \times 10^{-14}$ \\
14 & $7.90 \times 10^{-9}$  & 22 & $1.28 \times 10^{-14}$ \\
15 & $1.35 \times 10^{-9}$  & 23 & $1.26 \times 10^{-14}$ \\
16 & $2.14 \times 10^{-10}$ & 24 & $1.22 \times 10^{-14}$ \\
17 & $3.14 \times 10^{-11}$ & 25 & $9.57 \times 10^{-15}$ \\
\bottomrule
\end{tabular}
\end{table}
Table~\ref{tab:errors_vs_M_example1} reports the relative error of the $\star$-method for different values of the truncation parameter $M$ in Example~1, with matrix size 2500, Arnoldi parameter $k=35$, and $t_{\max}=4$.
 While the bound of Theorem~\ref{thm:bound} predicts the exponential convergence of the error, it predicts $M=260$ to reach machine accuracy when applied to Example~1. A similar behavior also appears when tested on other examples. 

\begin{figure}[ht]
\centering

\begin{subfigure}[b]{0.48\textwidth}
    \centering
    \includegraphics[width=\linewidth]{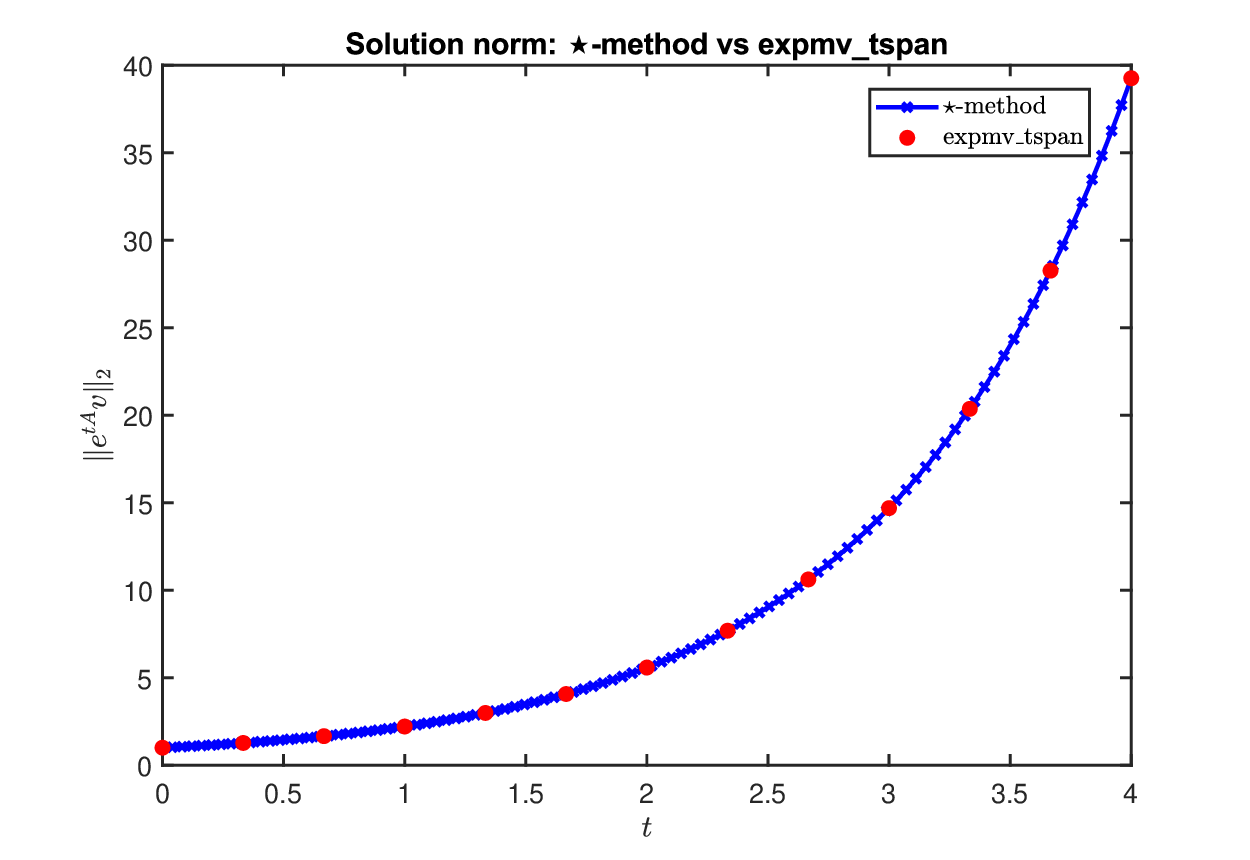}
    \caption{Example 4 ($n=20$, $M=12$, $k=19$)}
\end{subfigure}
\hfill
\begin{subfigure}[b]{0.48\textwidth}
    \centering
    \includegraphics[width=\linewidth]{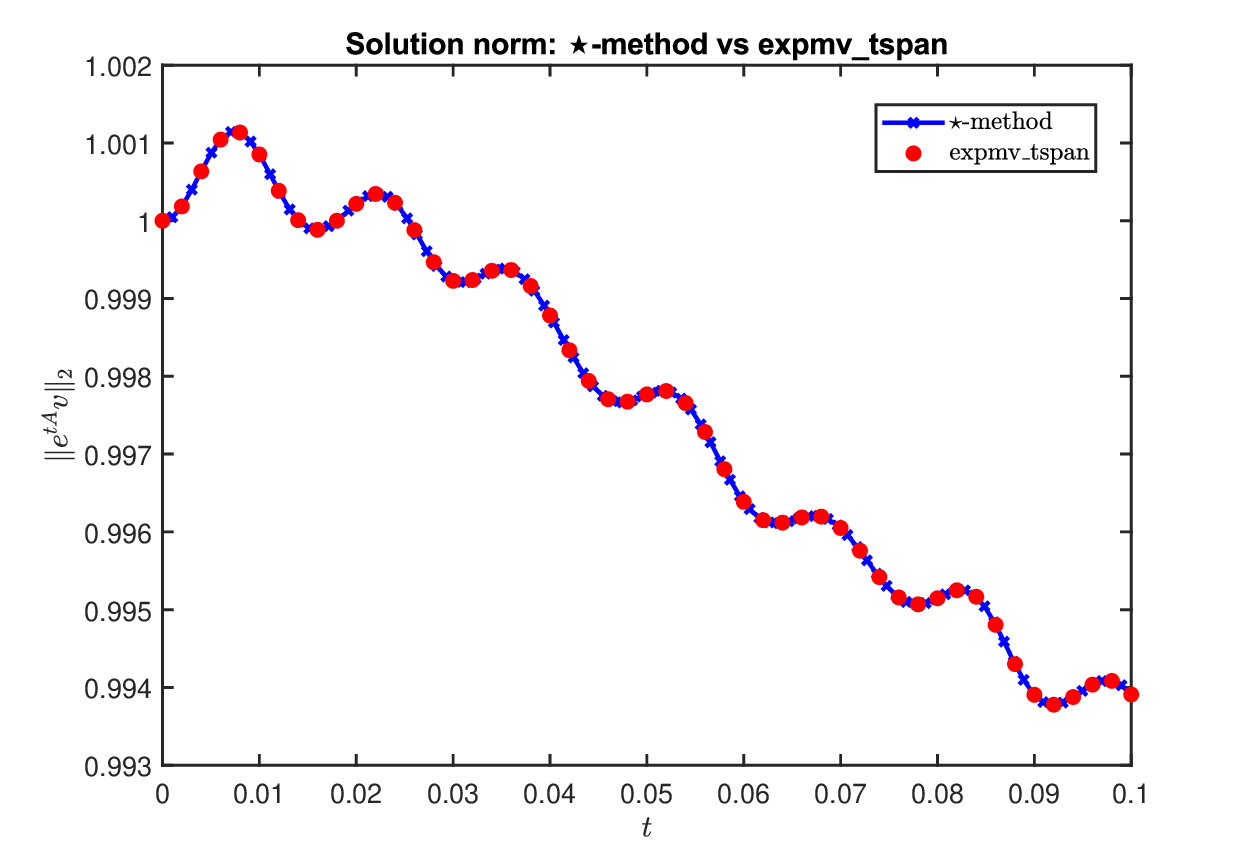}
    \caption{Example 8 ($n=841$, $M=50$, $k=70$)}
\end{subfigure}

\caption{Comparison of solution norms $\|e^{At}v\|$ for Examples~4 and~8
over their respective time intervals ($[0,4]$ and $[0,0.1]$).
The $\star$-method results are shown as curves and
\texttt{expmv\_tspan} values as discrete points overlaid on that curve.}

\label{fig:solution_norms}

\end{figure}
The solution obtained by the $\star$-approaches can be used to compute the exponential at every point of the given interval. Figure~\ref{fig:solution_norms} illustrates this by plotting the solution norm $\|e^{At}v\|$ as a blue curve generated by the proposed $\star$-method computed over 100 equispaced points on the interval, for Examples~4 and~8. Additionally, the corresponding solutions from \texttt{expmv\_tspan} are shown as discrete points (red circles) overlaid on that curve.

\newpage
\section{Conclusions}
\label{sec4}
In this paper, we presented a novel algorithm for computing the matrix exponential $e^{\tilde{A} t}v$ for all $t$ in a given interval, where $\tilde{A} $ is a matrix and $v$ a vector. The proposed method, referred to as the $\star$-method, represents the solution using a Legendre polynomial expansion of the function $e^{\tilde{A} t}v$, enabling efficient evaluation at any time point within the interval.
This approach builds upon a recent formulation for solving linear ODEs introduced in~\cite{giscard2015exact, pozza2023newk}, which is derived from the so-called $\star$-algebra. By this path, the matrix exponential approximation is given by solving a Stein matrix equation, which allows for the integration of Krylov subspace methods when handling large and sparse matrices.

Numerical experiments demonstrate that the $\star$-method, when combined with the Arnoldi algorithm, achieves accuracy and efficiency comparable to the method proposed in~\cite{al2011computing}, which approximates the matrix exponential on a set of a certain given number of time points. A key advantage of our method is its ability to produce the solution over the entire time interval, making it particularly suitable for applications where the desired evaluation times are too many or not known in advance.

 In Section~\ref{sec:trunc:err}, we have introduced a first upper bound for the truncation error of the $\star$-approach for the matrix exponential. Although the upper bound correctly predicts the exponential convergence of the error, its behavior in practical cases seems not yet satisfactory. Nevertheless, the approach described in this paper provides a much-needed starting point to develop sharper bounds by a better polynomial approximation.

Finally, we have reformulated the problem as the linear system \eqref{eq:linsyst}, which enables the development of new preconditioning techniques to evaluate the action of the matrix exponential on a vector. This is part of ongoing research.

\section*{Acknowledgments}
  We thank Fabio Durastante for his help with Examples 8--10 and the anonymous referee for their work and useful suggestions.
  This work was supported by Charles University Research programs No. PRIMUS/21/SCI/009 and UNCE/24/SCI/005, and by the Magica project ANR-20-CE29-0007 funded by the French National Research Agency. 

  \subsection*{Declaration of generative AI and AI-assisted technologies}
  \noindent   OpenAI’s ChatGPT (GPT‑5) and Microsoft 365 Copilot (bizchat.20260210.47.1) were used for language polishing and for assisting with bibliographic research, as well as with trivial mathematical derivations, all of which were strictly supervised and verified by the authors.


 \bibliography{cas-refs}
\bibliographystyle{acm}



\end{document}